\documentclass[11pt]{article}

\usepackage{amsmath,amssymb,amsthm,amsfonts,mathrsfs}
\usepackage[pdftex]{graphicx,epsfig,color}
\def\R{\mathbb R}

\numberwithin{equation}{section}

\newtheorem{theorem}{Theorem}

\newtheorem{lemma}[theorem]{Lemma}
\newtheorem{proposition}[theorem]{Proposition}
\newtheorem{corollary}[theorem]{Corollary}

\makeindex

\begin{document}

\title{Closing geodesics in $C^1$ topology}
\footnotetext{The author has been supported by the program ``Project ANR-07-BLAN-0361, Hamilton-Jacobi
et th\'eorie KAM faible''.}


\author{L.~Rifford\thanks{Universit\'e de Nice-Sophia
    Antipolis, Labo.\ J.-A.\ Dieudonn\'e, UMR CNRS 6621, Parc
    Valrose, 06108 Nice Cedex 02, France ({\tt
      Ludovic.Rifford@math.cnrs.fr})}}

\date{}



\maketitle

\begin{abstract}
Given a closed Riemannian manifold, we show how to close an orbit of the geodesic flow by a small perturbation of the metric in the $C^1$ topology.
\end{abstract}


\section{Introduction}
\label{Intro}

Given a dynamical system and a recurrent point $x$, the Closing Problem is concerned with the existence of a nearby dynamical system with a closed orbit through $x$. The statement of the Closing Problem for vector fields in the $C^r$ topology is as follows.\\

\noindent {\bf $C^r$-Closing Problem for vector fields.} Let $M$ be a smooth compact manifold, $r\geq 0$ an integer, $X$ be a vector field of class $C^{\max \{1,r\}}$ on $M$, and $x$ be a recurrent point of $X$. Does there exist a $C^r$ vector field $Y$ arbitrary close to $X$ in the $C^r$ topology so that $x$ is a periodic point of $Y$ ? \\

The answer to the Closing Problem in the $C^0$ topology is trivially affirmative (see \cite[\textsection 1 p. 958]{pugh67a}). The Closing  Problem in the $C^1$ topology is much more difficult. In the 60's, Charles Pugh \cite{pugh67a} solved by a tour de force  the Closing Problem in the $C^1$ topology.

\begin{theorem}[$C^1$-Closing Lemma for vector fields]
Let $M$ be a smooth compact manifold. Suppose that some vector field $X$ has a nontrivial recurrent trajectory through $x\in M$ and suppose that $\mathcal{U}$ is a neighborhood of $X$ in the $C^1$ topology. Then there exists $Y\in \mathcal{U}$ such that $Y$ has a closed orbit through $x$. 
\end{theorem} 

Since then, the Pugh $C^1$-Closing Lemma has been developed in several directions. Pugh himself \cite{pugh67b} extended it to the case of nonwandering points for vector fields, diffeomorphisms and flows. Then, in the 80's, Charles Pugh and Clark Robinson \cite{pr83} studied the Closing Problem for conservative dynamical systems such as the Hamiltonian systems. 

\begin{theorem}[Closing Lemma for Hamiltonian vector fields in the $C^2$ topology]
\label{THMpr}
Let $(N,\omega)$ be a symplectic manifold of dimension $2n\geq 2$ and $H : N \rightarrow \R$ be a given Hamiltonian of class $C^2$. Let $X$ be the Hamiltonian vector field associated with $H$ and $\phi^H$ the Hamiltonian flow. Suppose that $X$ has a nontrivial recurrent trajectory through $x\in N$ and suppose that $\mathcal{U}$ is a neighborhood of $X$ in the $C^1$ topology. Then there exists $Y\in \mathcal{U}$ such that $Y$ is a Hamiltonian vector field and $Y$ has a closed orbit through $x$. 
\end{theorem} 

Note that a perturbation of the Hamiltonian in the $C^2$ topology induces a perturbation of the associated Hamiltonian vector field in the $C^1$ topology only. We refer the reader to the exhaustive memoir \cite{arnaud98} of Marie-Claude Arnaud for a detailed presentation and proofs of various versions of the closing lemma as well as comments on the Closing Problem in the $C^2$ topology (almost nothing is known in that case). Knowing the Pugh-Robinson Closing Lemma for Hamiltonian vector fields (they prove actually Theorem \ref{THMpr} for nonwandering points), it is natural to ask what happens for geodesics flows. \\

\noindent {\bf $C^r$-Closing Problem for geodesic flows.}  Let $(M,g)$ be a smooth compact manifold, $r\geq 0$ an integer, and $(x,v)$ be fixed in the unit tangent bundle $U^gM$. If $(x,v)$ is recurrent with respect to the geodesic flow of $g$, do there exist smooth metrics arbitrary close to $g$ in the $C^r$ topology so that the unit speed geodesic starting at $x$ with initial velocity $v$ is periodic ?\\

For that problem, nothing is known. Even the $C^0$-Closing Lemma for geodesic flows is unproved (see \cite[\textsection 10 p. 309]{pr83}). Let us explain why in few words. A geodesic flow may indeed be viewed as an Hamiltonian flow on the cotangent bundle $N=T^*M$ equipped with the canonical symplectic form. Given a smooth Riemannian metric $g$, we may define a smooth  Hamiltonian $H:T^*M \rightarrow \R$ by (in local coordinates)
$$
H(x,p) = \frac{1}{2} \left(\|p\|_x^*\right)^2 \qquad \forall (x,p) \in T^*M,
$$
where $\|\cdot \|^*$ denotes the dual metric on $T^*M$. In that way, the Closing Problem for geodesic flows becomes a Closing Problem for Hamiltonian vector fields with a specific type of perturbation. As a matter of fact, a perturbation of a given metric in a small neighborhood $\Omega$ of some $x\in M$ induces a perturbation of the associated Hamiltonian in all the fibers $T_y^*M$ with $y\in \Omega$. However, in Theorem \ref{THMpr}, one allows perturbations of the Hamiltonian in both variables. In other words, in contrast to Theorem \ref{THMpr},  the perturbations allowed in the Closing Problem for geodesic flows cannot be localized in the phase space $T^*M$ but only in $M$. 

The aim of the present paper is to prove a closing lemma for geodesic flows in the $C^1$ topology on the metric, that is in the $C^0$ topology for the associated dynamics. To state the result, let us make clear the notations which will be used throughout the paper.

Let $M$ be a smooth compact manifold without boundary of dimension $n\geq 2$ (throughout the paper, smooth always means of class $C^{\infty}$). For every Riemannian metric $g$ on $M$ of class $C^k$ with $k\geq 2$, denote by $|v|_x^g$ the norm of a vector $v \in T_xM$, by $U^gM$ the unit tangent bundle, and by $\phi_t^g$ the geodesic flow on $U^gM$. Moreover, for every $(x,v) \in U^gM$, denote by  $\gamma_{x,v}^g : \R \rightarrow M$ the unit speed geodesic starting at $x$ with initial velocity $v$. The aim of the present paper is to show how to close an orbit of the geodesic flow with a small conformal perturbation of the metric in the $C^1$ topology. Pick a Riemannian distance on $TM$ and denote by $d_{TM} (\cdot,\cdot)$ the geodesic distance associated to it on $TM$. Note that since all Riemannian distances are Lipschitz equivalent on compact subsets, the choice of the metric on $TM$ is not important.  Our main result is the following:

\begin{theorem}\label{THM}
Let $g$ be a Riemannian metric on $M$ of class $C^k$ with $k\geq 3$ (resp. $k=\infty$), $(x,v)\in U^gM$   and $\epsilon >0$ be fixed. Then there exist a metric $\tilde{g} = e^f g$ with $f:M \rightarrow \R$ of class $C^{k-1}$ (resp. $C^{\infty}$) satisfying $ \| f \|_{C^1}<\epsilon$, and $\bigl(\tilde{x},\tilde{v}\bigr) \in U^{\tilde{g}}M$ with $d_{TM}\bigl(x,v),(\tilde{x},\tilde{v})\bigr) < \epsilon$, such that the geodesic $\gamma_{(\tilde{x},\tilde{v})}^{\tilde{g}}$ is periodic.
\end{theorem}

The idea of our proof is first to observe that thanks to the Poincar\'e recurrence  theorem, the geodesic flow is nonwandering on $U^gM$. Then we perform the construction of a connecting metric which preserves the transverse pieces of the geodesics crossing the box. This is done thanks to Lemma 5.\\

There is a constant $C>0$ such that if $(x,v) , \bigl(\tilde{x},\tilde{v}\bigr) \in TM$ satisfy $(x,v) \in U^gM$ and $d_{TM}\bigl(x,v),(\tilde{x},\tilde{v})\bigr) < \epsilon$ with $\epsilon >0$ small enough, then there is a smooth diffeomorphism $\Phi : M \rightarrow M$ such that 
$$
\Phi(x) = \Phi(\tilde{x}), \quad d \Phi (x,v) = \bigl( \tilde{x}, \tilde{v}\bigr), \quad \mbox{ and } \quad \|\Phi - Id\|_{C^2} < C \epsilon.
$$
Therefore, the following result is an easy consequence of Theorem \ref{THM}:

\begin{corollary}\label{COR}
Let $g$ be a Riemannian metric on $M$ of class $C^k$ with $k\geq 3$ (resp.  $k=\infty$), $(x,v)\in U^gM$ and $\epsilon >0$ be fixed.  Then there exists a metric $\tilde{g} $ of class $C^{k-1}$ (resp. $C^{\infty}$) with $\left\|\tilde{g} -g  \right\|_{C^1}<\epsilon$ such that the geodesic $\gamma_{(x,v)}^{\tilde{g}}$ is periodic.
\end{corollary}

The Pugh $C^1$-Closing Lemma has strong consequences on the structure of the flow of generic vector fields (see \cite[\textsection 1 p. 1010]{pugh67b}). It is worth noticing that our result is not striking enough to infer relevant properties for generic geodesic flows (for instance, the existence of an hyperbolic periodic orbit is not stable under $C^0$ perturbations on the dynamics). Such interesting properties would follow from the following conjecture which is tempting in view of Pugh's Closing Lemma. (We refer the reader to \cite{contrerasicm} and references therein for known generic properties of geodesic flows in the $C^2$ topology.)\\

\noindent {\bf Conjecture.}  Let $(M,g)$ be a smooth compact manifold and $(x,v)$ be fixed in the unit tangent bundle $U^gM$. There exist smooth metrics arbitrary close to $g$ in the $C^2$ topology so that the unit speed geodesic starting at $x$ with initial velocity $v$ is periodic.\\

In 1951, Lyusternik and Fet proved that at least one closed geodesic exists on every smooth compact Riemannian manifold (see \cite{klingenberg,lf51}). Our Corollary \ref{COR} shows that any pair $(x,v) \in U^gM$ may indeed be seen as a pair $\bigl(\gamma_k(0),\dot{\gamma}_k(0)\bigr)$ for some sequence of closed orbits $\{\gamma_k\}$  with respect to smooth Riemannian metrics $\{g_k\}$ converging to $g$ in the $C^1$ topology.  \\
 
\noindent  The paper is organized as follows: In Section \ref{Sconnecting}, we state and prove a result which is crucial to the proof of Theorem \ref{THM}. This result, Proposition \ref{PROPconnect}, shows how to connect two close geodesics while preserving a finite set of transverse geodesics, by a conformal perturbation of the initial metric with control on the support of the conformal factor and on its $C^1$ norm. Then, the proof of Theorem \ref{THM} is given in Section \ref{PROOFTHM} and the proofs of some technical results are postponed to the appendix. \\
 
 \noindent Notations: Throughout this paper, we denote by $\langle \cdot, \cdot\rangle$ the Euclidean inner product and by $|\cdot|$ the Euclidean norm in $\R^k$, and for any $x\in \R^k$ and any $r\geq 0$, we set $B^k(x,r):= \{y\in \R^k: |y-x|<r\}$. \\
 
\noindent {{\it Acknowledgements}: We are grateful to two anonymous referees for helpful remarks and suggestions.

\section{Connecting geodesics with obstacles}\label{Sconnecting}

\subsection{Statement of the result}\label{subconnecting}

Let $n\geq 2$ be an integer, $\tau >0$ be fixed, and let $\bar{g}$ be a complete Riemannian metric of class $C^k$ with $k\geq 3$ or $k=\infty$ on $\R^n$. Denote by $|v|_x^{\bar{g}}$ the norm with respect to $\bar{g}$ of a vector $(x,v) \in T\R^n = \R^n \times \R^n$, denote by  $\phi^{\bar{g}}_t$ the geodesic flow of $\bar{g}$ on $\R^n \times \R^n$ and for every $( x, v) \in \R^n \times \R^n$, denote by $\bar{\gamma}_{x,v}$ the geodesic with respect to $\bar{g}$ which starts at $x$ with velocity $v$. Assume that the curve $\bar{\gamma} :[0,\tau] \rightarrow \R^n$ is a geodesic with respect to $\bar{g}$ satisfying the following property ($e_1$ denotes the first vector in the canonical basis $(e_1, \ldots, e_n)$ of $\R^n$):
\begin{itemize}
\item[(A)] $ \left| \dot{\bar{\gamma}} (t) - e_1 \right| \leq 1/10,$  for every $t\in [ 0, \tau]$.
\end{itemize}
Set 
$$
\begin{array}{rl}
\bar{x}^0 & = \left(\bar{x}^0_1, \ldots, \bar{x}^0_n\right) := \bar{\gamma}(0), \quad \bar{v}^0  = \left(\bar{v}^0_1, \ldots, \bar{v}^0_n\right) := \dot{\bar{\gamma}}(0),\\
\bar{x}^{\tau} &  = \left(\bar{x}^{\tau}_1, \ldots, \bar{x}^{\tau}_n\right)   := \bar{\gamma} (\tau), \quad \bar{v}^{\tau}  = \left(\bar{v}^{\tau}_1, \ldots, \bar{v}^{\tau}_n\right) := \dot{\bar{\gamma}}(\tau).
\end{array}
$$
Our aim is to show that, given $(x,v), (y,w) \in \R^n \times \R^n$ with $|v|_{x}^{\bar{g}}= |w|_{y}^{\bar{g}}=  1$ sufficiently close to $\bigl( \bar{x}^0,  \bar{v}^0\bigr)$, there exists a Riemannian metric $\tilde{g}$ of class $C^{k-1}$ which is conformal to $\bar{g}$ and whose support and $C^1$-norm are controlled, which connects $(x, v)$ to  $\left(\bar{\gamma}_{y,w} ( \tau ), \dot{\bar{\gamma}}_{y,w} ( \tau) \right) = \phi^{\bar{g}}_{\tau} (y,w)$ and which preserves finitely many transverse geodesics. Set
 $$
  \mathcal{R} (\rho):= \Bigl\{ (t,z)  \, \vert \, t \in \left[ \bar{x}_1^0, \bar{x}_1^{\tau} \right],\, z \in B^{n-1}(0,\rho)  \Bigr\} \qquad \forall \rho >0.
 $$
Let us state our result.
  
\begin{proposition}\label{PROPconnect}
Let $\tau >0$ and $\bar{\gamma} :[0,\tau] \rightarrow \R^n$ satisfying assumption (A) be fixed. Let $\rho>0$ be such that $\bar{\gamma}\left( [0,\tau]\right) \subset   \mathcal{R} (\rho/2)$ be fixed. There are $\bar{\delta}= \bar{\delta} (\tau,\rho) \in (0,\tau/3)$ and $C= C(\tau,\rho) >0$ such that  the following property is satisfied: For every $(x,v), (y, w) \in U^{\bar{g}}\R^n$ satisfying
\begin{eqnarray} \label{Assconnect1}
\left|x -\bar{x}^0 \right|, \, \left|y -\bar{x}^0\right|, \, \left| v - \bar{v}^0 \right|, \,   \left| w - \bar{v}^0 \right| < \bar{\delta},
\end{eqnarray}
and for every finite set of unit speed geodesics 
$$
\bar{c}_1 \, : \, I_1= [a_1,b_1] \longrightarrow \R^n,  \quad \cdots, \quad \bar{c}_L  \, : \, I_L = [a_L,b_L] \longrightarrow \R^n
$$ 
satisfying 
\begin{eqnarray}\label{Assconnect2}
 \bar{c}_l(a_l),  \bar{c}_l (b_l)   \, \notin \mathcal{R} (\rho)  \qquad \forall l \in \{1, \ldots, L\},
\end{eqnarray}
\begin{eqnarray}\label{Assconnect3}
 \left( \bar{c}_l(s), \dot{\bar{c}}_l (s)\right)  \neq \phi_t^{\bar{g}} (x,v), \, \phi_t^{\bar{g}} (y,w) \qquad \forall l \in \{1, \ldots, L\}, \, \forall s\in I_l, \, \forall t\in [0,\tau],
\end{eqnarray}
\begin{eqnarray}\label{Assconnect4}
\mbox{and } \quad \left| \dot{\bar{c}}_l (s) - \dot{\bar{c}}_l (s') \right| < 1/8  \qquad \forall l \in \{1, \ldots, L\}, \, \forall s, s' \in I_l,
\end{eqnarray}
there are $\tilde{\tau} >0$ and a Riemannian metric $\tilde{g}= e^{f} \bar{g}$ on $\R^n$ with $f: \R^n \rightarrow \R$ of class $C^{k-1}$ (or $f$ of class $C^{\infty}$ if $\bar{g}$ is itself $C^{\infty}$) satisfying the following properties:
\begin{itemize}
\item[(i)] $\mbox{Supp } (f) \subset \mathcal{R}  (\rho) $;
\item[(ii)] $\|f  \|_{C^1} < C \,  \left| (x,v) - (y,w) \right| $;
\item[(iii)] $\left| \tilde{\tau}  - \tau \right| < C \, \left| (x,v) - (y,w) \right|$;
\item[(iv)] $\phi^{\tilde{g}}_{\tilde{\tau}} (x,v) = \phi^{\bar{g}}_{\tau} (y,w)$;
\item[(v)] for every $l\in \{1,\ldots, L\}$ $\bar{c}_l$ is, up to reparametrization,  a geodesic with respect to $\tilde{g}$. 
\end{itemize}
\end{proposition}

The proof of Proposition \ref{PROPconnect} occupies Sections \ref{SECwithout} to \ref{SECobstacles}. First, in Section \ref{SECwithout}, we restrict our attention to assertions (i)-(iv) by showing how to connect two unit speed geodesics in a constructive way (compare \cite[Proposition 3.1]{frclosing1} and \cite[Proposition 2.1]{frclosing2}). Then, in Section \ref{SECreparametrization}, we provide a lemma (Lemma \ref{LEMreparametrization}) which explains how a conformal factor may preserve geodesic curves. Finally, in Section \ref{SECobstacles}, we invoke transversality arguments together with Lemma \ref{LEMreparametrization} to conclude the proof of Proposition \ref{PROPconnect}.


\subsection{Connecting geodesics without obstacles}\label{SECwithout}

Let us first forget about assertion (v). For every $x\in \R^n$, denote by $\bar{G}(x)$ the $n\times n$ matrix whose coefficients are the $\bigl(\bar{g}_x\bigr)_{i,j}$, set $\bar{Q}:=\bar{G}^{-1}$ and define the Hamiltonian $\bar{H} :\R^n \times \R^n \rightarrow \R$ of class $C^k$ by
$$
\bar{H}(x,p) := \frac{1}{2} \left\langle p, \bar{Q}(x) p \right\rangle \qquad \forall x\in \R^n, \forall p \in \R^n.
$$
There is a one-to-one correspondence between the geodesics associated with $\bar{g}$   and the Hamiltonian trajectories of $\bar{H}$. For every $(x,v) \in \R^n \times \R^n$, the trajectory $\bigl( x(\cdot),p(\cdot)\bigr) : [0,\infty) \rightarrow \R^n \times \R^n$ defined by 
$$
\bigl( x(t),p(t)\bigr) := \Bigl( \bar{\gamma}_{x,v} (t), \bar{G} \bigl(\bar{\gamma}_{x,v}(t)\bigr) \, \dot{\bar{\gamma}}_{x,v} (t) \bigr)\Bigr) \qquad \forall t \geq 0,
$$
is the solution of the Hamiltonian system 
\begin{eqnarray}
\label{sysH}
\left\{ \begin{array}{rcl}
\dot{x}(t) & = & \frac{\partial  \bar{H}}{\partial p}  \bigl(x(t),p(t) \bigr)\\
\dot{p}(t) & = & - \frac{\partial  \bar{H}}{\partial x} \bigl(x(t),p(t) \bigr) 
\end{array}
\right.
\end{eqnarray}
such that $\bigl(x(0),p(0)\bigr) = \bigl(x, \bar{G}(x) \, v\bigr)$. Let $(x,v), (y,w) \in U^{\bar{g}}\R^n$ be fixed, set 
\begin{eqnarray}\label{notations}
x^0 := x, \quad p^0 := \bar{G}(x) \, v,  \quad x^{\tau} := \bar{\gamma}_{y,w} (\tau), \quad v^{\tau} :=   \dot{\bar{\gamma}}_{y,w} (\tau), \quad  p^{\tau} := \bar{G}(x^{\tau}) \, v^{\tau}.
\end{eqnarray}
Our aim is first to find a metric $\tilde{g}$ whose associated matrices $\tilde{G}, \tilde{Q}$ have the form 
$$
\tilde{G}(x)^{-1}= \tilde{Q}(x)= e^{-f(x)} \bar{Q}(x) \qquad \forall x \in \R^n,
$$
 in such a way that the trajectory $\bigl(x(\cdot),p(\cdot)\bigr): [0,\infty) \rightarrow \R^n \times \R^n $ of the Hamiltonian system 
\begin{eqnarray}
\label{sysHS}
\left\{ \begin{array}{rcl}
\dot{x}(t) & = & \frac{\partial  \tilde{H}}{\partial p}  \bigl(x(t),p(t) \bigr)\\
\dot{p}(t) & = & - \frac{\partial  \tilde{H}}{\partial x} \bigl(x(t),p(t) \bigr) 
\end{array}
\right.
\end{eqnarray}
associated with the new Hamiltonian $\tilde{H}=H_f :\R^n \times \R^n \rightarrow \R$ defined by
\begin{eqnarray}\label{HS}
\tilde{H}(x,p) = H_f (x,p) := \frac{1}{2} \left\langle p, \tilde{Q}(x) p \right\rangle = \frac{e^{-f(x)}}{2} \left\langle p, \bar{Q}(x) p \right\rangle  \quad \forall x\in \R^n, \forall p \in \R^n,
\end{eqnarray}
and starting at $\bigl(x^0,p^0)$ satisfies $\left(x (\tau ),p (\tau )\right)= \bigl(x^{\tau},p^{\tau}\bigr)$.  Note that for any $x, p \in \R^n$,
\begin{eqnarray}\label{Hf1}
 \frac{\partial  H_f}{\partial p}  (x,p) = \tilde{Q} (x) p  =  e^{-f(x)} \bar{Q} (x) \, p
\end{eqnarray}
and for every $i=1, \ldots, n$,
\begin{eqnarray}\label{Hf2}
 \frac{\partial H_f}{\partial x_i} (x,p) = \frac{1}{2} \left\langle p, \frac{\partial \tilde{Q}}{\partial x_i}  (x) \, p \right\rangle =  \frac{e^{-f(x)}}{2} \left\langle p, \frac{\partial \bar{Q}}{\partial x_i} (x) \, p \right\rangle -  \frac{1}{2} \left\langle p, \tilde{Q} (x) \, p \right\rangle \frac{\partial f}{\partial x_i}(x).
\end{eqnarray}

Let us fix a smooth function $\psi :[0,\tau] \rightarrow [0,1]$ satisfying
\begin{eqnarray*}
\psi(t) =0\quad \forall t\in [0,\tau/3] \quad \mbox{ and } \quad \psi(t)=1\quad \forall t  \in [2\tau/3,\tau].
\end{eqnarray*}
Given $(x,v), (y,w) \in U^{\bar{g}}\R^n$, we define a trajectory 
$$
\mathcal{X} \bigl(\cdot; (x,v),(y,w)\bigr) : [0,\tau] \, \longrightarrow \, \R^n
$$
of class $C^{k+1}$ by
\begin{eqnarray}\label{formulaX}
\mathcal{X}\bigl(t;(x,v),(y,w)\bigr) := \bigl(1-\psi (t)  \bigr) \, \bar{\gamma}_{x,v} (t) + \psi (t) \, \bar{\gamma}_{y,w} (t) \qquad \forall\, t \in [0,\tau].
\end{eqnarray}
We note that the mapping $\bigl(t, (x,v), (y,w)\bigr) \mapsto  \mathcal{X}\bigl(t;(x,v),(y,w)\bigr) $ is $C^{k+1}$ in the $t$ variable but only $C^{k-1}$ in the variables $x, v, y, w$. Let $\alpha \bigl( \cdot; (x,v), (y,w) \bigr) : [0,\tau] \rightarrow [0,+\infty)$ be the function defined as
\begin{multline*}
\alpha \bigl(t; (x,v), (y,w) \bigr)\\
 := \int_0^t \sqrt{ \left\langle \dot{\mathcal{X}} \bigl(s; (x,v),(y,w)\bigr), \bar{G} \left(   \mathcal{X} \bigl(s; (x,v),(y,w)\bigr) \right)   \, \dot{\mathcal{X}} \bigl(s; (x,v),(y,w)\bigr)   \right\rangle } \, ds,
 \end{multline*}
for every  $t \in [0,\tau]$. We observe that $\alpha \bigl( \cdot; (x,v), (y,w) \bigr)      $ is strictly increasing, of class $C^{k+1}$ in the $t$ variable, and of class $C^{k-1}$ in the variables $x,v,y,w$. Let
$$
\theta \bigl( \cdot; (x,v),(y,w)\bigr) : \left[0,\tilde{\tau}= \tilde{\tau}\bigl( (x,v), (y,w)\bigr):=\alpha \bigl(\tau; (x,v), (y,w)    \bigr)\right] \, \longrightarrow \, [0,\tau]
$$
denote its inverse, which is of class $C^{k+1}$ in $t$, $C^{k-1}$ in $x,v,y,w$, and satisfies (we set $\theta(\cdot) = \theta \left((\cdot; (x,v),(y,w)\right)$ and $\mathcal{X}(\cdot) = \mathcal{X} \left((\cdot; (x,v),(y,w)\right)$)
\begin{eqnarray*}
\dot{\theta} (s) = \frac{1}{ \sqrt{ \left\langle \dot{\mathcal{X}} \bigl( \theta (s) \bigr), \bar{G} \left(   \mathcal{X} \bigl(   \theta (s) \bigr) \right)   \, \dot{\mathcal{X}} \bigl(  \theta (s)\bigr)   \right\rangle }} \qquad \forall s\in [0,\tilde{\tau}].
\end{eqnarray*}
 Then, we define a new trajectory 
$$
\tilde{x} (\cdot) = \tilde{x} \bigl(\cdot; (x,v),(y,w)\bigr) : \left[0,\tilde{\tau}\bigl( (x,v), (y,w)\bigr)\right]  \, \longrightarrow \, \R^n
$$
of class $C^{k+1}$ by
\begin{eqnarray*}
\tilde{x} \bigl(t; (x,v),(y,w)\bigr) := \mathcal{X} \left( \theta (t)   \right) \qquad \forall t \in [0,\tilde{\tau}].
\end{eqnarray*}
 By construction, 
\begin{eqnarray}\label{3juillet1}
\left\{
\begin{array}{l}
 \tilde{x} (t) =  \mathcal{X} \bigl(t;(x,v),(y,w)\bigr) = \bar{\gamma}_{x,v} (t) \quad \forall t \in [0,\tau/3],\\
 \tilde{x}  (t) =  \mathcal{X} \bigl(t;(x,v),(y,w)\bigr) =  \bar{\gamma}_{y,w} (t) \quad  \forall t \in \left[\tilde{\tau} - \tau/3, \tilde{\tau} \right],
\end{array}
\right.
\end{eqnarray}
and
\begin{eqnarray*}
 \left\langle \dot{\tilde{x}} (t), \bar{G} \left(   \tilde{x} (t) \right)   \, \dot{\tilde{x}} (t)    \right\rangle =1\qquad \forall t \in [0,\tilde{\tau}].
 \end{eqnarray*}
 This means that the adjoint trajectory 
$$
 \tilde{p} (\cdot) = \tilde{p} \bigl(\cdot; (x,v),(y,w)\bigr) : \left[0,\tilde{\tau}\bigl( (x,v), (y,w)\bigr)  \right]  \, \longrightarrow \, \R^n
$$
 defined by
\begin{eqnarray}\label{DEFptilde}
 \tilde{p} \bigl(t; (x,v),(y,w)\bigr)  :=  \bar{G} \left(     \tilde{x} (t) \right) \,  \dot{\tilde{x}} (t) \qquad \forall t \in [0,\tilde{\tau}],
\end{eqnarray}
satisfies
\begin{eqnarray}\label{mercredi30a}
\dot{\tilde{x}} (t)  = \frac{\partial \bar{H}}{\partial p}  \left(    \tilde{x} (t),  \tilde{p} (t)  \right)   \qquad \forall t \in [0,\tilde{\tau}]
\end{eqnarray}
and 
\begin{eqnarray}\label{1juillet12}
\bar{H}  \left(    \tilde{x} (t),  \tilde{p} (t)   \right)   = \frac{1}{2}  \qquad \forall  \in [0,\tilde{\tau}].
\end{eqnarray}
 We now define the function 
$$
 \tilde{u} (\cdot)  = \Bigl(    \tilde{u}_1 \bigl(\cdot; (x,v),(y,w)\bigr), \ldots,  \tilde{u}_n \bigl(\cdot; (x,v),(y,w)\bigr)  \Bigr) : \left[0,\tilde{\tau}  \right]  \, \longrightarrow \, \R^n
$$
by
\begin{eqnarray}\label{DEFutilde}
 \tilde{u}_i (t)  := 2  \dot{\tilde{p}}_i (t) + \left\langle  \tilde{p} (t), \frac{\partial \bar{Q}}{\partial x_i} \left( \tilde{x} (t)\right)   \,  \tilde{p} (t)  \right\rangle \qquad \forall i=1, \ldots, n, \, \forall t\in  [0,\tilde{\tau}].
\end{eqnarray}
By construction, the function $\tilde{p}$ is of class $C^{k}$ in the $t$ variable, $\tilde{u}$ is $C^{k-1}$ in the $t$ variable, and all the functions $\tilde{\tau}, \tilde{p}, \tilde{u}$ are $C^{k-1}$ in the $x,y,v,w$ variables. Furthermore, it follows that
$$
\dot{\tilde{p}} (t)  = -  \frac{\partial \bar{H}}{\partial x}  \left(    \tilde{x} (t),  \tilde{p} (t)  \right)  + \frac{1}{2}  \tilde{u} (t)    \qquad \forall t \in [0,\tilde{\tau}],
$$
\begin{eqnarray*}
\left\{
\begin{array}{l}
\left( \tilde{x} (0), \tilde{p} (0) \right) = \left(x^0,p^0\right), \\
\left( \tilde{x}\bigl(\tilde{\tau}\bigr), \tilde{p} \bigl(\tilde{\tau}\bigr) \right) = \left( x^{\tau}, p^{\tau}\right),
\end{array}
\right.
\end{eqnarray*}
(using the notations (\ref{notations}) and remembering (\ref{3juillet1})), and 
\begin{eqnarray}\label{5juillet0}
\tilde{u} \bigl(t; (x,v),(y,w)\bigr) =0_n \qquad \forall t \in [0,\tau/3] \, \cup \, \left[\tilde{\tau} - \tau/3, \tilde{\tau} \right]
\end{eqnarray}
(by (\ref{3juillet1}), (\ref{DEFptilde}), and (\ref{DEFutilde})). Since $\bar{H}$ is of class $C^k$ with $k\geq 3$, the mapping 
\begin{multline*}   
\mathcal{Q}  \, : \, \left( (x,v), (y,w), s \right)  \in   \left( \R^n \times \R^n\right)  \times \left( \R^n  \times \R^n\right) \times [0,1]  \\
\, \longmapsto \,  \left( \tilde{\tau} \bigl( (x,v),(y,w)\bigr), \tilde{u}\bigl(s\tilde{\tau}  \bigl( (x,v),(y,w)\bigr); (x,v),(y,w) \bigr) \right)
\end{multline*}
is of class at least $C^1$. Therefore, since for all $(x,v) \in U^{\bar{g}}\R^n$ with $\left|x-\bar{x}^0\right| \leq 1$, 
$$
\mathcal{Q} \bigl( (x,v), (x,v), s \bigr)  = \left( \tau, 0  \right) \qquad  \forall\, s\in [0,1],
$$
there exists a constant $K >0$ such that, for every pair $(x,v), (y,w) \in U^{\bar{g}}\R^n$ with $\left|x-\bar{x}^0\right|, \left|y - \bar{x}^0\right| \leq 1$, 
\begin{eqnarray}\label{5juillet1}
\left| \tilde{\tau} \bigl( (x,v), (y,w) \bigr) - \tau \right| &\leq& \left|\mathcal{Q} \bigl( (x,v), (y,w) , 0 \bigr) -  \mathcal{Q} \bigl( (x,v), (x,v) , 0 \bigr)\right| \nonumber\\
 &\leq&  K  \left| (x,v) - (y,w) \right|,
\end{eqnarray}
and analogously
\begin{eqnarray}\label{5juillet2}
\bigl\| \tilde{u}  \bigl( \cdot; (x,v), (y,w) \bigr)  \bigr\|_{C^0} \leq K   \left| (x,v) - (y,w) \right|.
\end{eqnarray}
Furthermore, we notice that differentiating (\ref{1juillet12}) yields
$$
\left\langle \frac{\partial \bar{H}}{\partial x} \bigl( \tilde{x}(t), \tilde{p}(t)\bigr), \dot{\tilde{x}}(t)  \right\rangle +
\left\langle \frac{\partial \bar{H}}{\partial p} \bigl( \tilde{x}(t),\tilde{p}(t)\bigr), \dot{\tilde{p}} (t) \right\rangle= 0  \qquad \forall\, t \in \left[0,\tilde{\tau}\right],
$$
which together with  (\ref{mercredi30a}) and  (\ref{DEFutilde}) gives
\begin{eqnarray}\label{5juillet3}
\bigl\langle \tilde{u}(t), \dot{\tilde{x}}(t)\bigr\rangle = 0   \qquad \forall\, t \in \left[0,\tilde{\tau}\right].
\end{eqnarray}
In conclusion, for every $(x,v), (y,w) \in U^{\bar{g}}\R^n$ satisfying $\left|x-\bar{x}^0\right|, \left|y - \bar{x}^0\right| \leq 1$, the function 
$$
t \in \left[0,\tilde{\tau} \bigl( (x,v),(y,w)\bigr)\right] \, \longmapsto \, \Bigl( \tilde{x} \bigl(t; (x,v),(y,w) \bigr),  \tilde{p} \bigl(t; (x,v),(y,w) \bigr),  \tilde{u} \bigl(t; (x,v),(y,w) \bigr)\Bigr)
$$
satisfies for every $t\in \left[  0,\tilde{\tau} \bigl( (x,v),(y,w)\bigr)\right]$ and every $i=1,\ldots,n$,
\begin{eqnarray}\label{systildexp}
\label{syscontrol}
\left\{ \begin{array}{rcl}
\dot{\tilde{x}}(t) & = &   \bar{Q} \bigl(\tilde{x}(t)\bigr) \, \tilde{p}(t)   \\
\dot{\tilde{p}}_i(t) & = &  - \frac{1}{2} \left\langle \tilde{p}(t), \frac{\partial \bar{Q}}{\partial x_i} \bigl( \tilde{x}(t)\bigr) \, \tilde{p}(t) \right\rangle  -  \frac{1}{2} \left\langle \tilde{p}(t), \bar{Q} \bigl( \tilde{x}(t)\bigr) \, \tilde{p}(t) \right\rangle \tilde{u}_i(t),
\end{array}
\right.
\end{eqnarray}
and properties (\ref{5juillet1})-(\ref{5juillet3}) hold. In particular, taking the constant $K>0$ larger if necessary, (\ref{5juillet1})-(\ref{5juillet2}) and (\ref{systildexp}) together with Gronwall's Lemma imply that
\begin{eqnarray}\label{8fev99}
\left| \dot{\tilde{x}}(t) - e_1 \right| \leq K  \left| (x,v) - (y,w) \right|. \qquad \forall t \in [0,\tilde{\tau}].
\end{eqnarray}
The proof of the following lemma (taken from \cite{frclosing1}) is postponed to Section \ref{PROOFeasylemma}.

\begin{lemma}\label{easylemma}
Let $T, \beta, \mu \in (0,1)$ with $3\mu \leq \beta <T$, and let $y (\cdot), w(\cdot)  : [0,T] \rightarrow \R^n$ be two functions of class respectively at least $C^k$ and $C^{k-1}$ satisfying 
\begin{eqnarray}\label{easylemmahyp0}
\left| \dot{y}(t) - e_1 \right| \leq 1/5 \qquad \forall\, t \in [0,T],
\end{eqnarray}
\begin{eqnarray}\label{easylemmahyp1}
w(t) = 0_n \qquad \forall\, t \in [0,\beta] \cup [T-\beta, T],
\end{eqnarray}
\begin{eqnarray}\label{easylemmahyp2}
 \langle \dot{y}(t), w(t) \rangle  = 0 \qquad \forall\, t \in [0,T].
\end{eqnarray}
Then, there exist a constant $K$ depending only on the dimension and $T$, and a function $W: \R^n \rightarrow \R$ of class $C^k$ such that the following properties hold:
\begin{itemize}
\item[(i)] ${\rm Supp} (W) \subset \Bigl\{  y(t) + (0,z) \, \vert \, t \in [\beta/2,T-\beta/2], z \in B^{n-1}(0,\mu) \Bigr\}$;
\item[(ii)] $\|W\|_{C^1} \leq \frac{K}{\mu} \bigl\|w(\cdot) \bigr\|_{C^0}$;
\item[(iii)] $\nabla W (y(t)) = w(t)$ for every $t\in [0,T]$;
\item[(iv)] $W (y(t)) = 0$ for every $t\in [0,T]$.
\end{itemize}
\end{lemma}

Therefore   taking $\bar{\delta} \in (0,\tau/3)$ in (\ref{Assconnect1}) small enough, applying the above Lemma with $y(\cdot)=\tilde{x}(\cdot), w(\cdot)=\tilde{u}(\cdot), T=\tilde{\tau}, \beta = \tau/3,$ and $\mu>0$ small enough, and remembering assumption (A), that $\bar{\gamma}([0,\tau]) \subset \mathcal{R} (\rho/2)$,  (\ref{5juillet0}), (\ref{5juillet2})-(\ref{5juillet3}), and (\ref{8fev99}) yields a universal constant $C=C(\tau,\rho)>0$  and a function  $f:\R^n \rightarrow \R$ of class $C^k$ satisfying the following properties:
\begin{itemize}
\item[(a)] $\mbox{Supp } (f) \subset \mathcal{R} (\rho) $;
\item[(b)] $\|f  \|_{C^1} < C \,   \left| (x,v) - (y,w) \right|$;
\item[(c)] for every $t\in [0,\tilde{\tau}]$, $\nabla f\bigl( \tilde{x}(t)\bigr) = \tilde{u}(t)$;
\item[(d)]  for every $t\in [0,\tilde{\tau}]$, $f\bigl( \tilde{x}(t)\bigr) = 0$.
\end{itemize} 

Then, there is a one-to-one correspondence between the geodesics of $\tilde{g}:=e^f \bar{g}$ and the solutions of the Hamiltonian system (\ref{sysHS}) associated with $\tilde{H}=H_f$ given by (\ref{HS}). For every $t\in [0,\tilde{\tau}]$, by construction of $f$, the function $
\bigl( \tilde{x}(\cdot),\tilde{p}(\cdot)\bigr) : [0,\tilde{\tau}] \longrightarrow \R^n \times \R^n$ satisfies
\begin{eqnarray*}
\dot{\tilde{x}} (t)   = e^{- f \bigl( \tilde{x}(t)\bigr)} \bar{Q} \bigl(\tilde{x}(t)\bigr) \, \tilde{p}(t)
\end{eqnarray*}
and for every $i=1, \ldots, n$,
\begin{eqnarray*}
\dot{\tilde{p}}_i(t) =   - \frac{e^{-  f \bigl( \tilde{x}(t)\bigr)    }}{2} \left\langle \tilde{p}(t), \frac{\partial \bar{Q}}{\partial x_i} \bigl(\tilde{x}(t)\bigr) \, \tilde{p}(t) \right\rangle  -  \frac{e^{-  f \bigl( \tilde{x}(t)\bigr)  }}{2} \left\langle \tilde{p}(t), \bar{Q} \bigl(\tilde{x}(t)\bigr) \, \tilde{p}(t) \right\rangle \frac{\partial f}{\partial x_i} \bigl( \tilde{x}(t)\bigr).
\end{eqnarray*}
This means that $\tilde{x}(\cdot)$ is a geodesic on $[0,\tilde{\tau}]$ with respect to $\tilde{g}$ starting from $ \tilde{x}(0)= x^0=x$ with initial velocity $v = \bar{G}(x^0)^{-1} \, p^0 = \tilde{G}(x^0)^{-1} \, \tilde{p}(0)$ and ending at $\tilde{x}(\tau) = x^{\tau} $ with final velocity $v^{\tau}= \bar{G}(x^{\tau})^{-1} \, p^{\tau} = \tilde{G}(x^{\tau})^{-1} \, \tilde{p}(\tau).$ This proves assertions (i)-(iv) of Proposition \ref{PROPconnect}.

\subsection{One remark about reparametrization}\label{SECreparametrization}

The following result will be useful to insure that the geodesic curves $\bar{c}_l(I_l)$ are preserved.

\begin{lemma}\label{LEMreparametrization}
Let $\bar{c}: I =[a,b] \rightarrow \R^n$ be a unit speed geodesic with respect to $\bar{g}$, $\bar{f} :\R^n \rightarrow \R$ be a function of class at least $C^2$, and $ \bar{\lambda} :\R^n \rightarrow \R$ be  such that
\begin{eqnarray}\label{ASSflambda}
\nabla \bar{f} \left( \bar{c}(t)\right) = \bar{\lambda} (t) \bar{p}(t) := \bar{\lambda} (t) \, \bar{G} \left(\bar{c}(t)\right) \,   \dot{\bar{c}}(t) \qquad \forall t \in I,
\end{eqnarray}
where $\nabla \bar{f}$ denotes the gradient of $\bar{f}$ with respect to the Euclidean metric. Then up to reparametrization, $c$ is a unit speed geodesic with respect to the metric $e^{\bar{f}} \bar{g}$.
\end{lemma}

Of course, Lemma \ref{LEMreparametrization}  is a consequence of the fact that the gradient of $\bar{f}$ with respect to $\bar{g}$ at $\bar{c}(t)$ is always colinear with the velocity $\dot{\bar{c}}(t)$. Such a result could be found in textbooks of Riemannian geometry. For sake of completeness, we prove Lemma \ref{LEMreparametrization} with the Hamiltonian point of view.

\begin{proof}[Proof of Lemma \ref{LEMreparametrization}]
Define the function $\beta : I \rightarrow \R$  by 
\begin{eqnarray}\label{beta28juin}
\beta (t) := \int_0^t e^{\bar{f}\left( \bar{c}(s)\right)/2} \, ds \qquad \forall t \in I.
\end{eqnarray}
It is a strictly increasing function of class at least $C^3$ from $I$ to $\tilde{I} = [0,\tilde{\tau}] := \beta(I) $.  Denote by $\theta : \tilde{I} \rightarrow I$ its inverse. Note that $\theta$ is at least $C^3$ and satisfies 
\begin{eqnarray}\label{13juin}
\dot{\theta} (s) = e^{-\bar{f}\left( \bar{c}(\theta(s))\right)/2} \qquad \forall s \in \left[0,\tilde{\tau}\right].
\end{eqnarray}
Define $\tilde{c}, \tilde{p}: \tilde{I} \rightarrow \R^n $ by 
$$
\tilde{c}(s) := \bar{c} \bigl( \theta(s)\bigr) \quad \mbox{ and } \quad  \tilde{p} (s) :=  e^{\bar{f}\left( \tilde{c}(s)\right)/2 } \, \bar{p}(\theta(s))  \qquad \forall s \in \tilde{I}.
$$
The metric $\hat{g} := e^{\bar{f}} \bar{g}$ is associated with matrices $\hat{G}, \hat{Q}$ given by
$$
\hat{G}(x)^{-1}= \hat{Q}(x)= e^{-\bar{f}(x)} \bar{Q}(x) \qquad \forall x \in \R^n.
$$
Then, for every $s\in \tilde{I}$,  $\dot{\tilde{c}}(s)$ and $\tilde{p}(s)$ are given by 
$$
\dot{\tilde{c}} (s) = \dot{\theta} (s) \dot{\bar{c}} \bigl(\theta(s)\bigr) =  \dot{\theta} (s)  \bar{Q} \bigl( \bar{c} (\theta(s)\bigr) \, \bar{p} (\theta(s)) =  \hat{Q} \bigl( \tilde{c} (s)\bigr) \, \tilde{p} (s)
$$
and (using (\ref{13juin}))
\begin{eqnarray*}
\bigl(\dot{\tilde{p}}\bigr)_i (s) & = & \frac{d}{ds} \left(  e^{\bar{f}\left( \tilde{c}(s)\right)/2 } \right)  \bigl(\bar{p}\bigr)_i (\theta(s)) +    e^{\bar{f}\left( \tilde{c}(s)\right)/2 } \, \dot{\theta}(s) \,  \bigl(\dot{\bar{p}}\bigr)_i (\theta(s)) \\
& = & \frac{d}{ds} \left(  e^{ \bar{f}\left( \tilde{c}(s)\right)/2 } \right)  \bigl(\bar{p}\bigr)_i (\theta(s)) - \frac{1}{2}  \,  \left\langle \bar{p} \bigl( \theta(s)\bigr),  \frac{\partial \bar{Q}_i}{\partial x_i} \bigl( \bar{c}(\theta(s)) \,  \bar{p} \bigl( \theta(s)\bigr) \right\rangle \\
& = & \frac{d}{ds} \left(  e^{\bar{f}\left( \tilde{c}(s)\right)/2 } \right)  \bigl(\bar{p}\bigr)_i (\theta(s)) - \frac{ e^{-\bar{f}\left( \tilde{c}(s)\right) }}{2} \left\langle \tilde{p}(s),  \frac{\partial \bar{Q}_i}{\partial x_i} \bigl( \tilde{c}(s) \,  \tilde{p}  (s) \right\rangle,
\end{eqnarray*}
where the first term is equal to (using (\ref{ASSflambda}))
\begin{eqnarray*}
 \frac{d}{ds} \left(  e^{\bar{f}\left( \tilde{c}(s)\right)/2 } \right)  \bigl(\bar{p}\bigr)_i (\theta(s))  & = & \frac{  e^{\bar{f}\left( \tilde{c}(s)\right)/2 }  }{2} \left\langle  \nabla \bar{f} \bigl(\tilde{c}(s)\bigr), \dot{\tilde{c}}(s) \right\rangle \bigl(\bar{p}\bigr)_i (\theta(s)) \\
& = & \frac{1}{2} \, e^{\bar{f}\left( \tilde{c}(s)\right)/2 } \left\langle  \bar{\lambda} (\theta(s)) \, \bar{p} (\theta(s)) , \hat{Q} \bigl( \tilde{c} (s) \, \tilde{p} (s) \right\rangle \bigl(\bar{p}\bigr)_i (\theta(s)) \\
& = & \frac{1}{2} \,  \left\langle  \tilde{p} (s) ,    \hat{Q} \bigl( \tilde{c} (s)\bigr) \, \tilde{p} (s) \right\rangle  \Bigl(  \bar{\lambda} (\theta(s))   \bigl(\bar{p}\bigr)_i (\theta(s)) \Bigr)\\
& = & \frac{1 }{2} \, \left\langle \tilde{p} (s) ,    \hat{Q} \bigl( \tilde{c} (s)\bigr) \, \tilde{p} (s) \right\rangle \, \frac{\partial \bar{f}}{\partial x_i}  \bigl( \tilde{c} (s)\bigr).
\end{eqnarray*}
Remembering (\ref{Hf1})-(\ref{Hf2}) with $f=\bar{f}$ and $\tilde{Q}=\hat{Q}$, this proves that $\bigl( \tilde{c}(\cdot), \tilde{p}(\cdot)\bigr) : \tilde{I} \rightarrow \R^n \times \R^n$ is a trajectory of the Hamiltonian system associated with $\tilde{H}=H_{\bar{f}}$ and in turn concludes the proof of the lemma.
\end{proof}

\subsection{Dealing with obstacles}\label{SECobstacles}

We now proceed to explain how to modify our construction in order to get assertion (v) of Proposition \ref{PROPconnect}. We fix $(x,v),  (y,w) \in U^{\bar{g}}\R^n$ satisfying (\ref{Assconnect1}) and consider a finite set of unit speed geodesics 
$$
\bar{c}_1 \, : \, I_1  \longrightarrow  \R^n, \quad  \cdots, \quad  \bar{c}_L  \, : \, I_L \longrightarrow \R^n
$$ 
satisfying assumptions (\ref{Assconnect2})-(\ref{Assconnect3}). We  set
$$
\bar{\Gamma} := \bigcup_{l=1}^L \bar{c}_l  \left( I_l \right).
$$
The construction that we performed in the previous section together with transversality arguments yield the following result. (We recall that for any function $\tilde{u}(\cdot) :[0,\tilde{\tau}] \rightarrow \R^n$, ${\rm Supp} \bigl(\tilde{u} (\cdot) \bigr) $ denotes the closure of the set of $t\in [0,\tilde{\tau}]$ such that $\tilde{u}(t)=0$.)

\begin{lemma}\label{LEMTHOM}
Taking $\bar{\delta}>0$ in (\ref{Assconnect1}) small enough, there are a positive constant $C=C(\tau,\rho)$, $\tilde{\tau}= \tilde{\tau} \bigl( (x,v),(y,w)\bigr)>0$, a function 
$$
\left( \tilde{x}(\cdot), \tilde{p}(\cdot) \right) =  \left( \tilde{x} \bigl( \cdot; (x,v), (y,w)\bigr),  \tilde{p} \bigl( \cdot; (x,v), (y,w)\bigr) \right)   \,     : \, [0,\tilde{\tau}] \longrightarrow \R^n
$$
of class $C^k$, and a function 
$$
\tilde{u}(\cdot)  =  \tilde{u} \bigl( \cdot; (x,v), (y,w)\bigr)   \,     : \, [0,\tilde{\tau}] \longrightarrow \R^n
$$
of class $C^{k-1}$ satisfying  (\ref{5juillet3}), (\ref{systildexp}),
\begin{eqnarray}\label{6juillettau}
\left| \tilde{\tau} - \tau \right| < C \, \left| (x,v) - (y,w) \right|,
\end{eqnarray}
\begin{eqnarray}\label{SupportUTHOM}
{\rm Supp} \bigl(\tilde{u} (\cdot) \bigr) \subset [\tau/5, 4\tau/5 ],
\end{eqnarray}
\begin{eqnarray}\label{6juilletu}
\bigl\| \tilde{u} \bigr\|_{C^0} \leq  C \,   \left| (x,v) - (y,w) \right|, 
\end{eqnarray}
\begin{eqnarray}\label{14juinMIDI}
 \bigl(\tilde{x}(0), \tilde{p}(0)\bigr) =  \bigl(x^{0},p^{0}\bigr), \quad     \bigl(\tilde{x}(\tilde{\tau}), \tilde{p}(\tilde{\tau})\bigr) =  \bigl(x^{\tau},p^{\tau}\bigr),
\end{eqnarray}
such that the following properties are satisfied: 
\begin{itemize}
\item[(i)] the curve $\tilde{x}\left({\rm Supp} \bigl(\tilde{u} (\cdot) \bigr) \right)$ is transverse to $\bar{\Gamma}$;
\item[(ii)] the set $\mathcal{T}_{\tilde{u}} \subset {\rm Supp} \bigl(\tilde{u} (\cdot) \bigr) $ defined by 
$$
\mathcal{T}_{\tilde{u}} := \Bigl\{ t \in {\rm Supp} \bigl(\tilde{u} (\cdot) \bigr)  \, \vert \, \tilde{x}(t) \in \bar{\Gamma} \Bigr\}
$$
is empty.
\end{itemize}
\end{lemma}

\begin{proof}[Proof of Lemma \ref{LEMTHOM}]
Let us consider the trajectory 
$$
\mathcal{X} (\cdot) = \mathcal{X} \bigl(\cdot; (x,v),(y,w)\bigr) : [0,\tau] \, \longrightarrow \, \R^n
$$
of class $C^{k+1}$ defined by (\ref{formulaX}). Since $\mathcal{X}(\cdot)$ coincides respectively with $\bar{\gamma}_{x,v}$ and $\bar{\gamma}_{y,w}$ on the intervals $[0,\tau/3]$ and $[2\tau/3, \tau]$ and since the $\bar{c}_l$'s are unit speed geodesics satisfying (\ref{Assconnect3}), there are $t_1 \in (0,\tau/3)$, $t_2 \in (2\tau/3, \tau)$ and $\nu  \in (0,\tau/100)$ such that 
\begin{eqnarray}\label{notin}
\mathcal{X} (t) \notin \bar{\Gamma} \qquad \forall t\in \left[ t_1-\nu, t_1+\nu\right] \, \cup \, \left[ t_2-\nu, t_2+\nu\right]. 
\end{eqnarray}
Moreover, since $\mathcal{X}$ is a reparametrization of $\tilde{x}(\cdot)$ satisfying (\ref{8fev99}), we have 
$$
 \left| \dot{\mathcal{X}}(t) - e_1 \right| \leq K'  \left| (x,v) - (y,w) \right|. \qquad \forall t \in [0,\tau],
$$
for some positive constant $K'$. Then taking $\bar{\delta}>0$ in (\ref{Assconnect1})  small enough and remembering (\ref{Assconnect4}), to prove (i) it is sufficient to show that we can perturb the curve $\mathcal{X}([0,\tau])$ to make it transverse to all the geodesic curves $\bar{c}(I_l)$ verifying 
$$
\left| \dot{\bar{c}}_l(s) - e_1 \right| < 1/2 \qquad \forall s \in I_l=[a_l,b_l].
$$
Without loss of generality, we may assume that for each such curve (denote by $\mathcal{L}$ the set of such $l$), we have $\bigl(\bar{c}_l(a_l)\bigr)_1 \leq \bar{x}^0 $ and $\bigl(\bar{c}_l(b_l)\bigr)_1 \geq \bar{x}^{\tau}$ (remember (\ref{Assconnect2})). Let us parametrize both curves $\mathcal{X} (\cdot)$ and $\bar{c}_l(\cdot)$ by their first coordinates (where $l\in \mathcal{L}$ is fixed). Namely, there are two diffeomorphisms $\theta_1 :J_1=[\alpha,\beta] \rightarrow [0,\tau], \theta_2 : J_2=[\alpha',\beta'] \rightarrow I_l$   of class $C^{k+1}$ such that
\begin{eqnarray}\label{6juillet1}
\left(\bigl( \mathcal{X} \circ \theta_1 \bigr) (s)\right)_1 = s \quad \forall s \in J_1 \quad \mbox{ and } \quad \left( \bigl( \bar{c}_l \circ \theta_2 \bigr) (s)\right)_1 = s \quad \forall s \in J_2.
\end{eqnarray}
Extending $I_l$ if necessary, we may indeed assume that $J_1 \subset J_2$. Define the  function $h_l : I \rightarrow \R^{n}$ of class $C^{k+1}$ by
$$
h_l(s) :=   \bigl( \mathcal{X} \circ \theta_1 \bigr) (s) -  \bigl( \bar{c}_l \circ \theta_2 \bigr) (s)  \qquad \forall s \in J_1=[\alpha,\beta].
$$
Fix a smooth function $\psi: [0,\tau] \rightarrow [0,1]$ satisfying 
\begin{eqnarray}\label{6juilletpsi}
\psi (t) = 0 \quad \forall t \in \left[0,t_1 - \nu\right] \cup \left[t_2+\nu,\tau\right] \quad \mbox{ and } \quad \psi (t) = 1 \quad \forall t \in \left[t_1+\nu,t_2-\nu\right].
\end{eqnarray}
For every $\omega  \in \R^{n}$ with $\omega_1 =0 $, define the curve $\mathcal{X}_{\omega} : [0,\tau] \rightarrow \R^n$ by
$$
\mathcal{X}_{\omega}(t) :=   \mathcal{X}  (t) +  \psi (t   ) \, \omega \qquad \forall t \in [0,\tau].
$$
If $\mathcal{X}_{\omega}\left([0,\tau]\right)$ intersects $\bar{c}_l(I_l)$, then 
\begin{eqnarray*}
 0_n & = &  \mathcal{X}_{\omega} (t) -  \bar{c}_l  (s) \\
 & = &  \mathcal{X} (t) -  \bar{c}_l  (s) + \psi(t)\, \omega \\
& = & \left( \mathcal{X} \circ \theta_1 \right) \bigl(\theta_1^{-1}(t) \bigr) -  \left( \bar{c}_l  \circ \theta_2 \right) \bigl(\theta_2^{-1}(s) \bigr)    + \psi(t)\, \omega,
 \end{eqnarray*}
for some $t\in [0,\tau]$ and $s\in J_1$. Since $\omega_1=0$ and (\ref{6juillet1}) is satisfied, we must have $ \theta_1^{-1}(t)= \theta_2^{-1}(s)$, then we obtain 
 $$
  0_n = \left( \mathcal{X} \circ \theta_1 \right) \bigl(\theta_1^{-1}(t) \bigr) -  \left( \bar{c}_l  \circ \theta_2 \right) \bigl(\theta_1^{-1}(t) \bigr)    + \psi(t) \,\omega = h_l \bigl(\theta_1^{-1}(t) \bigr)  +  \psi(t) \, \omega.
  $$
  Furthermore, by (\ref{notin}),  if $\omega$ is small enough, the restriction of $\mathcal{X}_{\omega}(\cdot) $ to the two intervals  $ \left[ t_1-\nu, t_1+\nu\right]$ and  $\left[ t_2-\nu, t_2+\nu\right]$        cannot intersect $\bar{\Gamma}$. By (\ref{6juilletpsi}), we infer that 
 $$
 h_l \bigl(\theta_1^{-1}(t) \bigr)  +  \omega = 0_n \quad \mbox{ for some } t \in \left[ t_1 +\nu, t_2 - \nu\right].
 $$
By Sard's Theorem (see for instance \cite{eg}), almost every value of $h_l$ is regular. In addition, if $-\omega$ is a regular value of $h_l$, then $\dot{h}_l(s) \neq 0_n$ for all $s$ such that $h_l(s)=-\omega$. This shows that if $-\omega$ is a small enough regular value of $h_l$, then $\mathcal{X}_{\omega}  \left([t_1-\nu,t_2+\nu]\right)  $ is transverse to $\bar{c}_l( I_l)$. Finally, we observe that  
\begin{eqnarray}\label{6juilletddot}
\left\{ \begin{array}{l}
\dot{\mathcal{X}}_{\omega} (t) = \dot{\mathcal{X}}  (t) +  \dot{\psi} (t   ) \, \omega \\
\ddot{\mathcal{X}}_{\omega} (t) = \ddot{\mathcal{X}}  (t) +  \ddot{\psi} (t   ) \, \omega 
\end{array}
\right.
\qquad \forall t \in [0,\tau].
\end{eqnarray}
Then taking a small enough $\omega \in \R^n$ with $\omega_1=0$ such that $-\omega$  is a regular value for all the $h_l$'s and proceeding as in Section \ref{SECwithout} provides $\tilde{\tau}= \tilde{\tau} \bigl( (x,v),(y,w)\bigr)>0$ and a triple 
\begin{multline*}
\left( \tilde{x}(\cdot), \tilde{p}(\cdot), \tilde{u}(\cdot) \right) =  \left( \tilde{x} \bigl( \cdot; (x,v), (y,w)\bigr),  \tilde{p} \bigl( \cdot; (x,v), (y,w)\bigr), \tilde{u} \bigl( \cdot; (x,v), (y,w)\bigr) \right)    \\
                    \,     : \, [0,\tilde{\tau}] \longrightarrow \R^n
\end{multline*}
satisfying (\ref{5juillet3}), (\ref{systildexp}), and (\ref{14juinMIDI}). Moreover, $\tilde{\tau}$ is given by
$$
\tilde{\tau} := \int_0^{\tau} \sqrt{ \left\langle \dot{\mathcal{X}}_{\omega} (s), \bar{G} \left(   \mathcal{X}_{\omega} (s) \right)   \, \dot{\mathcal{X}}_{\omega} (s)     \right\rangle } \, ds
$$
and for every $t\in [0,\tilde{\tau}]$,
\begin{eqnarray*}
\tilde{u}(t) & = & 2\dot{\tilde{p}} (t)  + 2  \frac{\partial \bar{H}}{\partial x}  \left(    \tilde{x} (t),  \tilde{p} (t)  \right) \\
& = & 2 \frac{d}{dt} \left\{ \bar{G} \left(     \tilde{x} (t) \right) \,  \dot{\tilde{x}} (t) \right\} + 2  \frac{\partial \bar{H}}{\partial x}  \left(    \tilde{x} (t),  \tilde{p} (t)  \right).
\end{eqnarray*}
From (\ref{6juilletddot}) and (\ref{5juillet1})-(\ref{5juillet2}), we deduce that taking $\omega$ small enough yields (\ref{6juillettau}) and (\ref{6juilletu}) for some universal constant $C=C(\tau,\rho)>0$. All in all, this shows assertion (i).

To show assertion (ii),  replace the curve $\tilde{x}(\cdot)$ (which is a reparametrization of $\mathcal{X}_{\omega}$) by a piece of unit speed geodesic (with respect to $\bar{g}$) in a neighborhood of each $t \in \left[0,\tilde{\tau}\right]$ such that $\tilde{x}(t)\in \bar{\Gamma}$ and reparametrize it as in Section \ref{SECwithout}. Let us explain briefly how to proceed. Given $\bar{t} \in \left(0,\tilde{\tau}\right)$ such that $\tilde{x}(\bar{t}) \in   \bar{\Gamma}$ and $\lambda>0$, define   $\tilde{x}_{\lambda }(\cdot) : [0,\tilde{\tau}] \rightarrow \R^n$ a small perturbation of $\tilde{x} (\cdot)$ by 
$$
\tilde{x}_{\lambda} (t) := \varphi \left( \frac{t-\bar{t}}{\lambda}\right) \, \tilde{x} (t) + \left[ 1-\varphi \left( \frac{t-\bar{t}}{\lambda}\right) \right]\, \bar{\gamma}_{\tilde{x}(\bar{t}), \dot{\tilde{x}}(\bar{t})} \bigl( t - \bar{t}\bigr) \qquad \forall t \in \left[0,\tilde{\tau}\right],
$$
where $\varphi : \R \rightarrow [0,1]$ is a smooth function satisfying
$$
\varphi(t ) =1 \quad \forall t \in (-\infty,-1] \cup [1, +\infty) \quad \mbox{ and } \quad \varphi(t)=0 \quad \forall t \in [-1/2,1/2].
$$ 
We leave the reader to check that taking $\lambda >0$ small enough yields the desired result. 
\end{proof}

Proposition \ref{PROPconnect}  follows easily from the following  result whose technical proof is postponed to Appendix \ref{PROOFLEMf}.

\begin{lemma}\label{LEMf}
There are $C= C(\tau,\rho) >0$ and a function $f:\R^n \rightarrow \R$ of class $C^{k-1}$ such that the following properties are satisfied:
\begin{itemize}
\item[(i)] $\mbox{Supp } (f) \subset \mathcal{R} (\rho)$;
\item[(ii)] $\|f  \|_{C^1} < C \, \left| (x,v) - (y,w) \right| $;
\item[(iii)] for every $t\in [0,\tilde{\tau}]$, $\nabla f\bigl( \tilde{x}(t)\bigr) = \tilde{u}(t)$;
\item[(iv)] for every $l\in \{1,\ldots, L\}$ and every $s\in I_l$, , there is $\lambda_l(s)$ such that 
$$
\nabla f\bigl( \bar{c}_l(s)\bigr) = \lambda_l (s) \bar{p}_l(s) := \lambda_l (s) \, \bar{G} \bigl(\bar{c}_l(s)\bigr) \,   \dot{\bar{c}}_l(s).
$$
\end{itemize} 
\end{lemma}

\section{Proof of Theorem \ref{THM}}\label{PROOFTHM}

Let $\gamma = \gamma_{x,v} : \R \rightarrow M$ be the geodesic starting from $x$ with velocity $v \in U_x^gM$ and $\epsilon >0$ be fixed.  Let $\tau \in (0,1/20)$ be a small enough time such that the curve $\gamma_{x,v} ( [-10\tau,10\tau])$ has no self-intersection. There exist an open neighborhood $\mathcal{U}_{x}$ of $x$ and a smooth diffeomorphism
 $$
 \theta_{x}: \mathcal{U}_{x} \longrightarrow B^n(0,1) \quad \mbox{ with } \quad \theta_x (x) =0_n \quad \mbox{ and } \quad  
\frac{d}{dt} \Bigr(  \theta_{x} \circ \gamma_{x,v}\Bigr)  (0) = e_1.
$$
Set 
$$
\bar{\gamma} (t) :=  \theta_{x} \left( \gamma_{x,v} (t) \right) \quad \forall t \in [-10\tau,10\tau]
$$
and 
$$
\bar{x}^0 := \bar{\gamma} (0) =0_n, \quad \bar{v}^0 := \dot{\bar{\gamma}} (0) = e_1, \quad \bar{x}^{\tau} := \bar{\gamma} (\tau), \quad \bar{v}^{\tau} := \dot{\bar{\gamma}} (\tau).
$$
The metric $g$ is sent, via the smooth diffeomorphism $\theta_{x}$, onto a Riemannian metric $\bar{g}$ of class $C^k$ on $B^n(0,1)$. Without loss of generality, we may assume that $\bar{g}$ is the restriction to $B^n(0,1)$ of a complete Riemannian metric of class $C^k$ defined on  $\R^n$. Denote by $\phi_t^{\bar{g}}$ the geodesic flow on $\R^n \times \R^n$. Set 
 $$
 \mathcal{H}_0 := \Bigl\{ y=(y_1,\ldots, y_n) \in \R^n \, \vert \, y_1=0  \Bigr\}.
 $$
Since $\bar{\gamma}(0)=0_n$ and $\dot{\bar{\gamma}}(0)=e_1$, taking $\tau$ smaller if necessary we may assume that  
\begin{eqnarray}\label{30juin1}
\mbox{and } \quad \left| \frac{d}{dt} \Bigr(  \theta_{x} \circ \gamma_{x,v}\Bigr)  (t) -e_1 \right| \leq 1/10 \qquad \forall t \in [0,\tau].
\end{eqnarray}
Keeping the notations of Section \ref{subconnecting}, we may also assume that there is $\rho>0$ such that the following properties are satisfied:
\begin{itemize}
\item[(i)] $\bar{\gamma} (t) \in  \mathcal{R} (\rho/2)= \Bigl\{ (t,z)  \, \vert \, t \in \left[ 0, \bar{x}_1^{\tau} \right],\, z \in B^{n-1}(0,\rho/2)  \Bigr\} \subset B^n(0,1),$
\item[(ii)] for every unit speed geodesic $\bar{c} \, : \, I= [a_1,b_1] \longrightarrow \R^n$ with $\bar{c}(I) \subset \mathcal{R} (2\rho) \subset B^n(0,1)  $, there holds
$$
 \left| \dot{\bar{c}}_l (s) - \dot{\bar{c}}_l (s') \right| < 1/8  \qquad  \forall s, s' \in I.
$$
\end{itemize}
Then, we can apply Proposition \ref{PROPconnect} to the curve $\bar{\gamma} : [0,\tau] \rightarrow \R^n$. Consequently, there are $\bar{\delta}= \bar{\delta} (\tau,\rho) \in (0,\tau/3)$ and $C= C(\tau,\rho) >0$ such that the property stated in Proposition \ref{PROPconnect} is satisfied. Define the section $\mathcal{S} \subset TM$ by 
$$
\mathcal{S} := d\theta_x^{-1} \left( \mathcal{H}_0 \times \R^n \right).
$$

Since $M$ is assumed to be compact and the geodesic flow preserves the Liouville measure, the Poincar\'e recurrence theorem implies that the geodesic flow is nonwandering on $U^gM$. Thus,  for every neighborhood $\mathcal{V}$ of $(x,v)$ in $U^gM$, there exist $t \geq 1$ and $(x',v') \in \mathcal{V}$ such that $\phi_t^g (x',v') \in \mathcal{V}$. Then, since $\gamma_{x,v}$ is transverse to $\mathcal{S}$ at time zero, for every $r>0$ small, there exist $(x^r,v^r), (x^r_*,v^r_*) \in \mathcal{S} \cap U^gM$, $T^r>0$ and $y^r, y^r_*, w^r, w^r_* \in B^n(0,1)$ such that
\begin{itemize}
\item[(a)] $(x^r_*,v^r_*) = \phi_{T^r}^g (x^r,v^r)$.
\item[(b)] $(y^r,w^r) = d\theta_x (x^r,v^r), (y_*^r,w_*^r) = d\theta_x (x_*^r,v_*^r)$;
\item[(c)] $(y^r,w^r), (y_*^r,w_*^r) \in U^{\bar{g}}\R^n$;
\item[(d)] $y^r, y^r_* \in \mathcal{H}_0$;
\item[(e)] $ \left|x -\bar{x}^0 \right|, \, \left|y -\bar{x}^0\right|, \, \left| v - \bar{v}^0 \right|, \,   \left| w - \bar{v}^0 \right| < \bar{\delta}$;
\item[(f)] $\left| (y^r,w^r) - (y_*^r,w_*^r)\right| <r$.
\end{itemize}
Recall that the cylinder $ \mathcal{R} (\rho/2)$ is defined by
$$
 \mathcal{R} (\rho/2):= \Bigl\{ (t,z)  \, \vert \, t \in \left[ 0, \bar{x}_1^{\tau} \right],\, z \in B^{n-1}(0,\rho/2)  \Bigr\} \subset B^n(0,1).
$$
The intersection of the curve $\gamma_{x^r,v^r} \left( [5\tau,T^r-5\tau]\right) $ with the open set $\theta_x^{-1} \left(  \mathcal{R} (\rho/2)\right)$ can be covered by a finite number of connected curves. More precisely, there are a finite number of unit speed geodesic arcs 
$$
\bar{c}_1 \, : \, I_1= [a_1, b_1] \longrightarrow B^n(0,1),  \quad \cdots, \quad  \bar{c}_L  \, : \, I_L=[a_L,b_L] \longrightarrow B^n(0,1)
$$ 
such that the following properties are satisfied:
\begin{itemize}
\item[(g)] For every $l\in \{1, \ldots, L\}$, $\bar{c}_l(a_l), \bar{c}_l (b_l) \in  \mathcal{R} (2\rho)  \setminus  \mathcal{R} (\rho/2)$;
\item[(h)] there are disjoint closed intervals $\mathcal{J}_1, \ldots, \mathcal{J}_L \subset  \left[ -5\tau,T^r-5\tau\right]$ such that
$$
\gamma_{x^r,v^r} (\mathcal{J}_l) \subset \mathcal{U}_{x},  \quad \bar{c}_l (I_l) = \theta_{x} \left( \gamma_{x^r,v^r}(\mathcal{J}_l)\right) \qquad \forall l=1, \ldots, L,
$$
$$
\mbox{ and } \quad \Bigl( \theta_x \left( \gamma_{x^r,v^r} \left( [5\tau,T_r-5\tau]  \right) \cap \mathcal{U}_{x} \right) \cap  \mathcal{R} (\rho/2) \Bigr) \,  \subset \,  \bigcup_{l=1}^L \bar{c}_l  ( I_l).
$$
\end{itemize}
From the above properties and (ii), we can connect $(y^r_*,w^r_*)$ to $\phi_{\tau}^{\bar{g}} (y^r,w^r)$ by preserving the curves $\bar{c}_1(I_1), \ldots, \bar{c}_L(I_L)$. We define the metric $\tilde{g}$ on $M$ by 
$$
\tilde{g} = \left\{ \begin{array}{l}
\tilde{g} \mbox{ on } M \setminus \mathcal{U}_{x}\\
\theta_{x}^* \left( e^{f}\bar{g}\right)  \mbox{ on } \mathcal{U}_{x}.
\end{array}
\right.
$$
We leave the reader to check that by construction the geodesic starting from $x^r_*$ with initial velocity $v^r_*$ is periodic. Taking $r>0$ small enough yields $d_{TM} \bigl( (x,v), (x_*^r,v_*^r)\bigr) < \epsilon$ and  $\|f\|_{C^1}<\epsilon$.

\appendix

\section{Proof of Lemmas  \ref{easylemma} and \ref{LEMf}}

\subsection{Proof of Lemma \ref{easylemma}}\label{PROOFeasylemma}

Define the function $\Phi : [0,T] \times \R^{n-1} \rightarrow \R^n$ by 
$$
\Phi (t,z)  := y( t) + (0,z) \qquad \forall \,(t,z) \in [0,T] \times \R^{n-1}.
$$
We can easily check that, thanks to (\ref{easylemmahyp0}), $\Phi$ is a diffeomorphism of class $C^k$ from $[0,T] \times \R^{n-1}$ into $\left[y_1(0),y_1(\tau)\right] \times \R^{n-1}$ which sends the cylinder
$[\beta/2,T-\beta/2] \times B^{n-1}\bigl(0,\mu\bigr)$ into the ``cylinder'' 
$$
\mathcal{C}_y (\mu) :=  \Bigl\{  y(t) + (0,z) \, \vert \, t \in [\beta/2,T-\beta/2], z \in B^{n-1}(0,\mu) \Bigr\},
$$
 and which satisfies
\begin{eqnarray*}
\|\Phi\|_{C^1}, \bigl\| \Phi^{-1}\|_{C^1} \leq K_0,
\end{eqnarray*}
for some positive constant $K_0$ depending on $T$ only. Define the function $\tilde{w} (\cdot) :[0,T] \rightarrow \R^n$ by
\begin{eqnarray*}
\tilde{w}(t) := \bigl(d\Phi \bigl(t,0_{n-1}\bigr)\bigr)^* \bigl( w(t) \bigr) \qquad \forall \,t\in [0,T].
\end{eqnarray*}
The function $\tilde{w}$ is $C^{k-1}$; in addition, by (\ref{easylemmahyp1}) and (\ref{easylemmahyp2}), it follows that
\begin{eqnarray*}
\tilde{w}(t) =0_n \qquad \forall t \in [0,\beta] \cup [T-\beta,T] \quad \mbox{ and } \quad  \tilde{w}_1(t)=0 \qquad \forall t \in [0,T].
\end{eqnarray*}
Let $\psi:\R \rightarrow [0,1]$ be an even
function of class $C^{\infty}$ satisfying the following properties:
\begin{itemize}
\item[-] $\psi(s)=1$ for $s\in [0,1/3]$;
\item[-] $\psi (s)=0$ for $s\geq 2/3$;
\item[-] $|\psi(s)|, |\psi'(s) | \leq 10$ for any  $s \in [0,+\infty)$.
\end{itemize}
Extend the function $\tilde{w}(\cdot)$ on $\R$ by $\tilde{w}(t):=0$ for $t\leq 0$ and $t \geq T$, and define the function $\tilde{W}:[0,T] \times \R^{n-1} \rightarrow \R$ by 
$$
\tilde{W} (t,z) = \psi \left( \frac{|z|}{\mu} \right) \left[  \sum_{i=2}^{n} \int_0^{z_i} \tilde{w}_{i}(t+s) ds \right] \qquad \forall\, (t,z) \in [0,T] \times \R^{n-1}.
$$
Since $\tilde{w}$ is $C^{k-1}$, $\psi$ is $C^k$, and $\tilde{W} (t,z)$ can be written as 
$$
\tilde{W} (t,z) = \psi \left( \frac{|z|}{\mu} \right) \left[  \sum_{i=2}^{n} \int_{t}^{t+z_i} \tilde{w}_{i}(t+s) ds \right], 
$$
it is easy to check that $\tilde{W}$ is of class $C^k$. Moreover, (using that $3\mu\leq \beta <T$) we check easily that
\begin{eqnarray*}
\mbox{Supp } \left(\tilde{W}\right) \subset [\beta/2,T -\beta/2] \times B^{n-1}\bigl(0,2\mu/3\bigr),
\end{eqnarray*}
\begin{eqnarray*}
\nabla \tilde{W} (t,0) = \tilde{w}(t), \quad \tilde{W}(t,0) =0  \qquad \forall t \in [0,T],
\end{eqnarray*}
and that
 (see the proof of \cite[Lemma 3.3]{frclosing1})
\begin{eqnarray*}
 \left\| \tilde{W} \right\|_{C^1} \leq  \frac{K_1}{\mu} \bigl\|\tilde{w} (\cdot)\bigr\|_{C^0},
 \end{eqnarray*}
 for some constant $K_1>0$. Finally, define the function $W:  \R^n \rightarrow \R$ by 
$$
W(x) := \left\{ 
\begin{array}{ll}
\tilde{W}\bigl( \Phi^{-1} (x)\bigr) & \mbox{ if } x \in \mathcal{C}_y (\mu)\\
0 & \mbox{ otherwise.}
\end{array}
\right.
$$
It is easy to see that $W$ satisfies (i)-(iv).

\subsection{Proof of Lemma \ref{LEMf}}\label{PROOFLEMf}
We proceed in several steps.\\

\noindent Step 1: Applying Lemma \ref{easylemma}, we get  a universal  constant $C_1= C_1(\tau,\rho)>0$ and a function $f_1:\R^n \rightarrow \R$ of class $C^{k}$ such that the following properties are satisfied:
\begin{itemize}
\item[$(i)_1$] $ \mbox{Supp } (f_1) \subset   \mathcal{R}  (2\rho/3)$;
\item[$(ii)_1$] $\bigl\| f_1\bigr\|_{C^1} < C_1 \, \left| (x,v) -(y,w) \right|$;
\item[$(iii)_1$]  $\nabla f_1\bigl( \tilde{x}(t)\bigr) =\tilde{u}(t),$ for every $t\in [0,\tilde{\tau}]$;
\item[$(iv)_1$] $f_1\bigl( \tilde{x}(t)\bigr) = 0,$ for every $t\in [0,\tau]$.
\\
\end{itemize}

\noindent Step 2:  Let $x_1, \ldots, x_N$ be a set of points in $ \mathcal{R}  (2\rho/3)$ such that 
$$
\left( \bigcup_{k, l=1, k\neq l}^L \Bigl(  \bar{c}_k(I_k) \cap \bar{c}_{l} (I_{l}) \Bigr)  \right) \cap  \, \mathcal{R}  (2\rho/3)= \Bigl\{ x_1, \ldots, x_N \Bigr\}.
$$
Note that by Lemma \ref{LEMTHOM} (ii), the set $\{x_1,\ldots, x_N\}$ does not intersect the curve $\tilde{x}\left( \mbox{Supp } (\tilde{u}(\cdot)\right)$. Let $\mu>0$ be such that the $N$ balls $B^n (x_1,2 \mu), \ldots, B^n (x_N,2\mu)$ are disjoint and do not intersect neither the curve $\tilde{x}\left( \mbox{Supp } (\tilde{u}(\cdot)\right)$ nor the boundary of $ \mathcal{R}  (2\rho/3)$. Define the $C^{k}$ function $f_2:\R^n \rightarrow \R$ by
$$
f_2(x)  := f_1 \left( \sum_{k=1}^N \left[ \psi \left( \frac{\bigl| x- x_k\bigr|}{3 \mu}\right) \, x_k + \left( 1 -  \psi \left( \frac{\bigl| x- x_k\bigr|}{3 \mu}\right)  \right) x  \right] \right) \qquad \forall x \in \R^n.
$$
By construction, there is a universal constant $C_2= C_2(\tau,\rho)>0$ such that $f_2$ satisfies the following properties:
\begin{itemize}
\item[$(i)_2$] $ \mbox{Supp } (f_2) \subset  \mathcal{R}  (2\rho/3) $;
\item[$(ii)_2$] $\bigl\| f_2\bigr\|_{C^1} < C_2 \, \left| (x,v) -(y,w) \right|$;
\item[$(iii)_2$]  $\nabla f_2\bigl( \tilde{x}(t)\bigr) =\tilde{u}(t),$ for every $t\in [0,\tilde{\tau}]$;
\item[$(iv)_2$]  $f_2\bigl( \tilde{x}(t)\bigr) = 0,$ for every $t\in [0,\tilde{\tau}]$;
\item[$(v)_2$] $f_2(x) = f_1(x)$ for every $x\in \R^n \setminus \Bigl( \bigcup_{k=1}^N B^n \bigl( x_k, 2 \mu \bigr) \Bigr)$;
\item[$(vi)_2$]Ê$\nabla f_2(x) =0$ for every $x\in \bigcup_{k=1}^N B^n \bigl( x_k, \mu \bigr).$\\
\end{itemize}

\noindent Step 3: Let $t_1, \ldots, t_K \in [0,\tau]$ be the set of times such that
 $$
 \tilde{x}\bigl( \mbox{Supp } (\tilde{u}(\cdot)\bigr) \cap \left( \bigcup_{l=1}^L \bar{c}_{l} (I_{l})\right)  = \Bigl\{ \tilde{x}  (t_k) \, \vert \, k=1, \ldots K \Bigr\}.
$$
Taking $\mu>0$ smaller if necessary, we may assume that  the balls $B^n \bigl( \tilde{x}  (t_1), 5 \mu\bigr)$, $\ldots, B^n \bigl( \tilde{x}  (t_K), 5\mu\bigr)$ are disjoint, do not intersect the boundary of $ \mathcal{R}  (\rho/2)$, and such that $\tilde{u}(t)=0$ for every $t\in [0,\tilde{\tau}]$ with $\tilde{x}(t) \in \bigcup_{k=1}^Q  B^n \bigl( \tilde{x}  (t_k), 5\mu\bigr)$ (remember Lemma \ref{LEMTHOM} (ii)). Set 
$$
\Omega :=  \bigcup_{k=1}^Q  B^n \bigl( \tilde{x}  (t_k), 2\mu\bigr).
$$
Taking $\mu>0$ smaller if necessary again, the projection (with respect to the Euclidean metric) $\mathcal{P}_0: \Omega \rightarrow \R^n$ to the set 
$$
S:= \bigcup_{k=1}^K  \Bigl( B^n \bigl( \tilde{x}  (t_k), 2\mu\bigr) \cap \tilde{x} \bigl( [0,\tilde{\tau}] \bigr)\Bigr),
$$
is of class $C^{k-1}$, has a $C^1$ norm $\bigl\| \mathcal{P}_0\bigr\|_{C^1}$ which is bounded by a universal constant, and satisfies  
$$
\mathcal{P}_0(x) = x \qquad \forall x \in S,
$$
$$
\mathcal{P}_0(x)  \in S \qquad \forall x \in \Omega,
$$
$$
\bigl| x -\mathcal{P}_0 (x) \bigr | < \frac{\mu}{2} \qquad \forall x \in \bigcup_{k=1}^K  \Bigl( B^n \bigl( \tilde{x}  (t_k), \mu/2 \bigr) \Bigr).
$$
Define the $C^{k-1}$ function $f_3: \R^n \rightarrow \R$  by
$$
f_3(x) := \left\{ \begin{array}{l}
 f_2 \Bigl( h(x) \mathcal{P}_0 (x) + \bigl( 1-  h(x) \bigr) x     \Bigr) \mbox{ if } x \in \Omega  \\
 f_2(x) \mbox{ otherwise,}
 \end{array}
 \right.
$$
where $h: \Omega \rightarrow \R$ is defined by
$$
h(x) :=  \psi \left( \sum_{q=1}^Q \frac{2\bigl| x- \tilde{x}(t_q)\bigr|}{3\mu} \right)  \qquad \forall x \in \Omega.
$$
We note that $h(x)=1$ for every $x\in \bigcup_{k=1}^K  \Bigl( B^n \bigl( \tilde{x}  (t_k), \mu/2\bigr) \Bigr)$ and $h(x)=0$ for every $x \in \Omega$ which does not belong to the set $ \bigcup_{k=1}^K  \Bigl( B^n \bigl( \tilde{x}  (t_k), \mu\bigr) \Bigr)$. Consequently, by construction, there is a universal constant $C_3=C_3(\tau,\rho) >0$ such that $f_3$ satisfies the following properties:
\begin{itemize}
\item[$(i)_3$] $\mbox{Supp } \bigl( f_3 \bigr) \subset  \mathcal{R}  (2\rho/3)  $;
\item[$(ii)_3$]Ê$\bigl\|f_3\bigr\|_{C^1} \leq C_3 \, \left| (x,v) -(y,w) \right|$;
\item[$(iii)_3$] $\nabla f_3 \bigl( \tilde{x}(t)\bigr) = \tilde{u}(t)$, for every $t\in [0,\tilde{\tau}]$;
\item[$(iv)_3$] $f_3 \bigl( \tilde{x}(t)\bigr) = 0$, for every $t\in [0,\tilde{\tau}]$;
\item[$(v)_3$] $f_3(x) = f_2(x)$ for every $x\in \R^n \setminus \Omega$;
\item[$(vi)_3$] $\nabla f_3(x) =0$ for every $x\in \bigcup_{k=1}^K B^n \bigl( \tilde{x}(t_k), \mu/2 \bigr)$.\\
\end{itemize}
 
\noindent Step 4: Denote by $d_{\bar{g}}:\R^n \times \R^n \rightarrow \R$ the Riemannian distance with respect to the Riemannian metric $\bar{g}$. Denote by $\mbox{dist}_{\bar{g}}^{\bar{\Gamma}} (\cdot)$ the distance function (with respect to $\bar{g}$) to the set $\bar{\Gamma}$. For every $\delta>0$, let $\mathcal{S}_{\delta}  \subset \mathcal{R}  (2\rho/3+\delta)  $ be the subset of $\bar{\Gamma}$ defined by 
$$
\mathcal{S}_{\delta} := \Bigl( \bar{\Gamma} \, \cap \,  \mathcal{R}  (\tau,2\rho/3+\delta) \Bigr)  \setminus \left(  \bigcup_{k=1}^N B^n \bigl( x_k, \mu/2 \bigr) \, \cup \, \bigcup_{k=1}^K B^n \bigl( \tilde{x}(t_q), \mu/4 \bigr) \right).
$$
For every $\delta, \mu>0$, we denote by $\mathcal{S}_{\delta}^{\mu}$ the open set of points whose distance (with respect to $\bar{g}$) to $\mathcal{S}_{\delta}$ is strictly less than $\mu$. There are $\delta, \mu>0$ such that  the function $\mbox{dist}_{\bar{g}}^{\bar{\Gamma}} (\cdot)$ is of class $C^k$ on $\mathcal{S}_{\delta}^{\mu}$, the projection $\mathcal{P}_{\bar{g}}^{\bar{\Gamma}} $ to $\bar{\Gamma}$ with respect to $\bar{g}$ is $C^{k-1}$ on $\mathcal{S}_{\delta}^{\mu}$, and both $\bigl\|  \mbox{dist}_{\bar{g}}^{\bar{\Gamma}} (\cdot)\bigr\|_{C^1(\mathcal{S}_{\delta}^{\mu})}, \bigl\|  \mathcal{P}_{\bar{g}}^{\bar{\Gamma}} (\cdot)\bigr\|_{C^1(\mathcal{S}_{\delta}^{\mu})}$ are bounded by a universal constant. Define the  function $f: \R^n \rightarrow \R$  by
$$
f(x) := \left\{ \begin{array}{l}
 f_3 \left(P(x)  \right) \mbox{ if } x \in \mathcal{S}_{\delta}^{\mu} \\
 f_3(x) \mbox{ otherwise,}
 \end{array}
 \right.
$$
where the mapping $P: \mathcal{S}_{\delta}^{\mu} \rightarrow \R^n$ is defined by
$$
P(x) := \psi \left(\frac{2 \mbox{dist}_{\bar{g}}^{\bar{\Gamma}}(x)}{3\mu} \right)  \mathcal{P}_{\bar{g}}^{\bar{\Gamma}} (x) + \left( 1 - \psi  \left( \frac{2\mbox{dist}_{\bar{g}}^{\bar{\Gamma}} (x)}{3\mu} \right) \right) x \qquad \forall x \in  \mathcal{S}_{\delta}^{\mu}.
$$
We leave the reader to check that if $\mu>0$ is small enough, the function $f$ is of class $C^{k-1}$ and satisfies assertions (i)-(iv) of Lemma \ref{LEMf} for some universal constant $C=C(\tau,\rho)>0$.

\end{document}